\documentclass[12pt,a4paper]{article}
\usepackage{graphicx}
\usepackage{amsmath}
\usepackage{amssymb}
\usepackage{amsfonts}
\usepackage{amsthm}
\usepackage{hyperref}
\numberwithin{equation}{section}

\renewcommand\Im{\mathop{\rm Im}}

\def\func#1{\mathop{\rm #1}}
\newcommand\ipart[1]{{\left\lfloor{#1}\right\rfloor}}
\newtheorem{theorem}{Theorem}[section]
\newtheorem{definition}[theorem]{Definition}
\newtheorem{lemma}[theorem]{Lemma}
\newtheorem{proposition}[theorem]{Proposition}
\newtheorem{remark}[theorem]{Remark}
\begin{document}

\title{Unique resonant normal forms for area preserving maps at
an elliptic fixed point}

\author{Vassili Gelfreich${}^1$\footnote{The author thanks
Dr.~N.~Br\"annstr\"om for the help with preparation of the manuscript.
This work was partially supported by a grant from the Royal Society.}
\and
Natalia Gelfreikh${}^2$
\\[24pt]
${}^1$\small Mathematics Institute, University of Warwick,\\
\small Coventry, CV4 7AL, UK\\
\small E-mail: v.gelfreich@warwick.ac.uk\\[12pt]
${}^2$\small Department of Higher Mathematics and Mathematical Physics,\\
\small Faculty of Physics, St.Petersburg State University, Russia\\
\small E-mail: gelfreikh@mail.ru}

\maketitle

\begin{abstract}
We construct a resonant normal form for an area-preserving map
near a generic resonant elliptic fixed point.
The normal form is obtained by a simplification of a formal interpolating
Hamiltonian. The resonant normal form is unique and therefore provides
the formal local classification for area-preserving maps with
the elliptic fixed point.  The total number of formal invariants is infinite.
We consider the cases of weak (of order $n\ge5$) and strong
(of order $n=3,4$) resonances.
We also construct unique normal forms for analytic families of area-preserving maps.

We note that our constructions involve non-linear grading functions.
\end{abstract}


\newpage

\section{Introduction}

Normal forms provide an important tool for the study of dynamical systems (see
e.g.~\cite{AKN,Kuznetsov}) as they can be used to achieve a substantial simplification of local
dynamics. The normal form theory uses canonical changes of variables to transform a map into a
simpler one called a normal form. Usually the normal form is more symmetric than the original map
and (at least in the setup of the current paper) is integrable.

Classical normal forms are not always unique and their further reduction has been studied by
various methods in order to derive unique normal forms~\cite{BS1992,BF2007,KOW1996,CWY2005}. Our
approach uses a grading function in order to simplify manipulations with formal series in several
variables. Unlike \cite{KOW1996} our grading functions are mostly non-linear.

In this paper we derive the unique normal forms for
area-preserving maps at a generic elliptic
resonant fixed point.
The symmetry of the normal form is induced by the Jacobian matrix of the map
and strongly depends on the type of the fixed point.
Let's review some results of the classical normal form theory.
The leading orders of the normal forms
are described in the classical book \cite{AKN}.

\medskip

Let $F_0:\mathbb{R}^{2}\to \mathbb{R}^{2}$ be an area-preserving map
with a fixed point at the origin
\begin{equation*}
F_{0}(0)=0.
\end{equation*}%
Since $F_{0}$ is area-preserving $\det DF_{0}(0)=1$. Therefore
we can denote the two eigenvalues of the Jacobian matrix
$DF_{0}(0)$ by $\mu $ and $\mu ^{-1}$.
These eigenvalues are often called {\em multipliers\/}
of the fixed point.  A fixed point is called
\begin{itemize}
\item {\em hyperbolic\/} if $\mu\in\mathbb{R}$ and $\mu\ne\pm1$;
\item {\em elliptic\/} if $\mu\in\mathbb{C}\setminus\mathbb{R}$ (in this case $|\mu|=1$);
\item {\em parabolic\/} if $\mu=\pm1$.
\end{itemize}
We also note that the word ``parabolic" is often used
for a generic fixed point with $\mu=1$.

\textbf{1. Hyperbolic fixed point.}
From the viewpoint of the normal form theory this case is
exceptional since the map can be transformed into
its Birkhoff normal form by an analytic change of variables.
Indeed, Moser \cite{Moser1956,Siegel1957} proved that
there is a canonical ({\em i.e.,\/} area- and orientation-preserving)
analytical change of variables such that in a neighbourhood
of the origin $F_0$ takes the form $(u,v)\mapsto(u_1,v_1)$ where
\begin{equation*}
\left\{
\begin{array}{l}
u_{1}=a(uv)u \\[6pt]
v_{1}=\displaystyle\frac{v}{a(uv)}
\end{array}
\right.
\end{equation*}
and $a$ is an analytic function of a single variable $uv$. Moreover,
$a(0)=\mu$.
It is interesting to note that although the change
of the variables is not unique the function $a$ is
uniquely defined by the map $F_0$. The normal form is integrable
since $u_1v_1=uv$. Therefore a neighbourhood of the fixed point
is foliated into invariant lines as shown in Figure \ref{Fig:BNF_Hyp}.
\begin{figure}
\includegraphics[width=\textwidth]{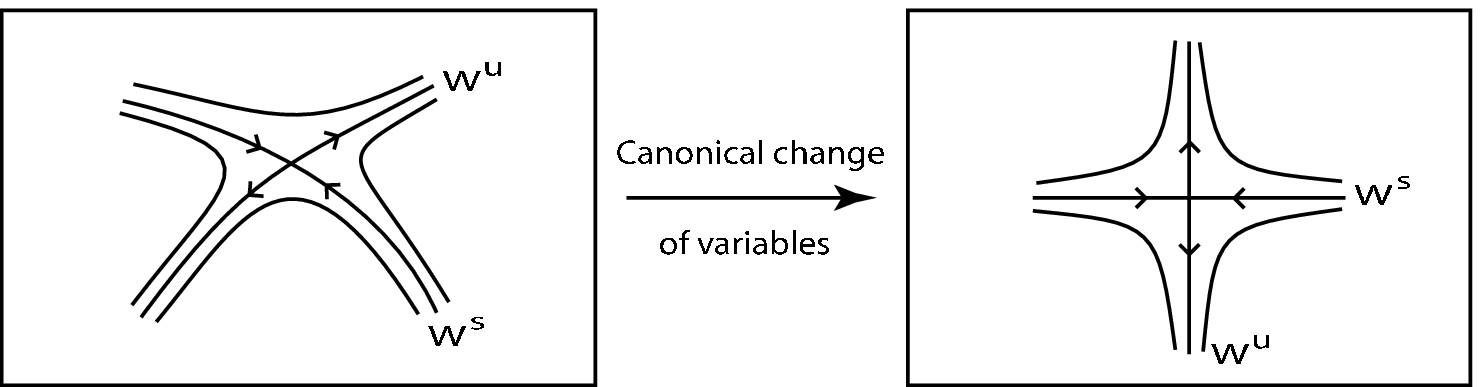}
\caption{Normal form coordinates in a neighbourhood of a hyperbolic fixed point.\label{Fig:BNF_Hyp}}
\end{figure}
Moreover,
in a neighbourhood of the fixed point $F_0$ coincides with the time-one map
of an analytic Hamiltonian system with one degree of freedom:
\[
F_0=\Phi^1_H\,.
\]
In the coordinates $(u,v)$ the Hamiltonian $H$ is a function of the product $uv$ only
and is given explicitly by the following integral:
\[
H(uv)=\int^{uv}\log a \,.
\]
We note that $H$ is a local integral only and hence does
not imply integrability of the map $F_0$. On the other hand
it provides a powerful tool for studying dynamics of area-preserving maps
and is very useful in the problems related to separatrices
splitting (see e.g. \cite{FS1990}).

\medskip

\textbf{2. Elliptic fixed point.} As the map is real-analytic the second
multiplier $\mu ^{-1}=\mu^*$, where $\mu^*$ is the complex conjugate of $\mu$.
Consequently the multipliers of an elliptic fixed point
belong to the unit circle $|\mu|=1$. Note the assumption $\mu\notin\mathbb{R}$
excludes $\mu =\pm 1$.
There is a linear area-preserving change of variables
such that the Jacobian of $F_0$ takes the form of
a rotation:
\begin{equation}\label{Eq:rotation}
DF_{0}(0)=R_{\alpha }=\left(
\begin{array}{cc}
\cos \alpha & -\sin \alpha \\
\sin \alpha & \cos \alpha%
\end{array}%
\right)
\end{equation}
where the rotation angle $\alpha$ is related to the multiplier: $\mu=e^{i\alpha}$.
The following theorem is a classical result
of the normal form theory.

\begin{theorem}\label{Thm:brnf}
(Birkhoff normal form) There exists a formal
canonical change of variables $\Phi $ such that
the map
\begin{equation*}
N=\Phi \circ F_0\circ \Phi ^{-1}
\end{equation*}%
commutes with the rotation $R_\alpha$:
\begin{equation*}
N\circ R_{\alpha }=R_{\alpha }\circ N.
\end{equation*}
\end{theorem}

The map $N$ is called a {\em Birkhoff normal form\/} of $F$ (see e.g.
\cite{AKN,Birkhoff1927}).
The map $R_{-\alpha}\circ N$ is tangent to identity, {\em i.e.}, its
Taylor series starts with the identity map. A tangent to identity
map can be formally represented as a time-one
map of an autonomous Hamiltonian system: there
is a formal Hamiltonian $H$ such that
\begin{equation}
N=R_\alpha\circ \Phi_H^1\,,
\end{equation}
where $\Phi_H^1$ is a formal time-one map. For the sake of completeness
we will provide a proof of this statement later.
The Hamiltonian inherits the symmetry of the normal form:
\begin{equation}
H\circ R_\alpha=H\,.
\end{equation}
The symmetry of the normal form has an important corollary:
 $H$ is a formal integral of $N$:
\[
H\circ N=H\circ R_\alpha\circ \Phi_H^1=H\circ \Phi_H^1=H\,.
\]
Therefore the Birkhoff normal form $N$ is integrable.
Coming back to the original variables we
obtain a formal integral of the original map.

Usually the series involved in the construction
of the normal form do not converge.
On the other hand it is possible to construct an analytic change of
variables which transforms the map into the normal form
up to a reminder of an arbitrarily high order.
In other words,
for any $p>0$ there is a canonical analytic change of coordinates
$\tilde \Phi_p$ such that its Taylor series
coincides with the formal series $\Phi$ up to the order $p$.
Then Taylor series of
$\tilde N_p=\tilde\Phi_p\circ F_0\circ\tilde\Phi_p^{-1}$
coincides with the formal series $N$ up to the order $p$,
and the map $\tilde N_p$
is in the normal form up to a reminder of order $p+1$, {\em i.e.,}
\[
\tilde N_p\circ R_\alpha-R_\alpha\circ\tilde N_p=O(r^{p+1})
\]
where $r=\sqrt{x^2+y^2}$. Moreover,
\[
\tilde N_p=\Phi_{H_p}^1+O(r^{p+1})
\]
where $H_p$ is a polynomial Hamiltonian obtained by
neglecting all terms of orders higher than $p$
in the formal series $H$.

Alternatively, it is possible to construct a smooth ($C^\infty$)
change of variables such that the remainders are flat.

The transformation to the normal form is not unique.
Traditionally this freedom is used to eliminate
some coefficients from the normal form map $N$.
We will take a slightly different point of
view and simplify the series for
the formal interpolating Hamiltonian.
In this way the problem is reduced to a study of a normal form
for a Hamiltonian system with symmetry. We note that this
is a classical subject
and lower order normal forms can be found in the literature
(see for example~\cite{GS1987,BGSTT1993}).

Our goal is to achieve a substantial simplification for all orders
and to derive the unique normal form for the generic case.
The results depend on the rotation angle $\alpha$ defined in (\ref{Eq:rotation}).

\begin{definition}
A fixed point is called {\em resonant} if there exists $n\in \mathbb{N}$ such that
$\mu ^{n}=1$. The least positive $n$ is called the {\em order of the resonance}.
If the fixed point is not resonant we call it\/ {\em non-resonant}.
A resonant fixed point is called {\em strongly resonant} if
$n\le4$. Otherwise it is called {\em weakly resonant}.
\end{definition}

%If the fixed point is resonant we may assume
%$$
%\alpha=\alpha _{n}=\frac{2\pi k}{n}
%$$
%for some integer $k$, $1\le k<n$.
The resonances of orders one and two are related
to a parabolic and not elliptic fixed point.

It is convenient to introduce the symplectic polar coordinates
$(I,\varphi)$ by
\begin{eqnarray*}
x&=&\sqrt{2I}\cos\varphi\,,\\
y&=&\sqrt{2I}\sin\varphi\,.
\end{eqnarray*}
If the fixed point is not resonant the normal form is a rotation:
\[
N=R_{\omega(I)}\qquad
\mbox{where}\qquad \omega(I)=\alpha+\sum_{k\ge1}\omega_kI^k\,.
\]
Note that the rotation angle depends on the action $I$.
The coefficients $\omega_k$ are defined uniquely and provide a
full set of formal invariants for $F_0$. Writing down Hamiltonian
equations and comparing their solution with the map $R_{-\alpha}\circ N=R_{\omega(I)-\alpha}$,
we can easily check that the formal interpolating
Hamiltonian is defined by
\[
\frac{\partial H}{\partial I}=\omega(I)-\alpha\,,
\qquad
\frac{\partial H}{\partial \varphi}=0
\,.
\]
Therefore it has the form
\[
H(I,\varphi)=\sum_{k\ge1}\frac{\omega_k}{k+1}\,I^{k+1}\equiv I^2 A(I)\,.
\]
We see that the coefficients of the series $A(I)$
are defined uniquely and can be used to formally classify the
maps instead of $\omega_k$.

In the case of a resonant fixed point the leading order of
the Hamiltonian $H$ has the form~\cite{AKN}
\[
H(I,\varphi)=I^2 A(I)+I^{n/2}B(I)\cos n\varphi +O(I^{1+n/2})
\]
where $A$ and $B$ are polynomial in $I$. In this paper we will show that
this form can be extended up to all orders and study the uniqueness
of the coefficients. The following theorem is the main result of the paper.

\begin{theorem}\label{Thm:main}
If $F_0$ is a smooth ($C^\infty$ or analytic) area preserving map with a resonant
elliptic fixed point at the origin, then there is a formal Hamiltonian $H$ and
formal canonical change of variables which conjugates $F_0$ with $R_\alpha\circ\Phi^1_H$.
Moreover, $H$ has the following form:
\begin{itemize}
\item
if $n\ge4$ and $A(0)B(0)\ne0$
\begin{equation}\label{Eq:unf}
H(I,\varphi)=I^2 A(I)+I^{n/2}B(I)\cos n\varphi
\end{equation}
where
\begin{equation}\label{Eq:ABn}
A(I)=\sum_{k\ge0}a_kI^k\,,\qquad
B(I)=\sum_{k\ge0}b_kI^{2k}\,.
\end{equation}
\item
If $n=3$ and $B(0)\ne0$
\begin{equation}\label{Eq:u3f}
H(I,\varphi)=I^3 A(I)+I^{3/2}B(I)\cos 3\varphi
\end{equation}
where
\begin{equation}\label{Eq:AB3}
A(I)=\sum_{\substack{k\ge0\\ k\ne 2\pmod 3}}a_kI^k\,,\qquad
B(I)=\sum_{\substack{k\ge0\\ k\ne 2\pmod 3}}b_kI^k\,.
\end{equation}
\end{itemize}
The coefficients of the series $A$ and $B$ are defined uniquely by the map $F_0$
provided the leading order is normalised to ensure $b_0>0$.
\end{theorem}

Note that the sign of $b_0$ changes after a substitution $\varphi\mapsto\varphi+\frac\pi n$.
In the theorem the change of the variables is note unique.

The theorem provides a formal classification for the generic maps
with respect to formal canonical changes of variables.

Theorem~\ref{Thm:main} follows from Propositions~\ref{Pro:n5}, \ref{Pro:n4} and~\ref{Pro:n3}
which state equivalent results in complex coordinates for $n=3$, $n=4$ and $n\ge5$ respectively.

\subsection*{Unique normal forms for analytic families}
Now instead of an individual area-preserving map $F_0$
we consider an analytic family of area-preserving maps $F_{\varepsilon }$.
We assume that the origin is an elliptic fixed point of $F_0$.
The implicit function theorem implies that there is $\varepsilon_0>0$ such that
in a neighbourhood of the origin
$F_\varepsilon$ has an elliptic fixed point for all $|\varepsilon|<\varepsilon_0$.
Without loosing in generality we can assume that the fixed point has already been
moved to the origin.

Let $\alpha$ be the rotation angle defined by $DF_0(0)$.
It is well known that the family can be transformed to the normal form
in a way similar to an individual map.
In other words there is a formal Hamiltonian $\chi_\varepsilon$ such
that
\begin{equation*}
N_{\varepsilon }=\Phi _{\chi _{\varepsilon }}^{-1}\circ F_{\varepsilon
}\circ \Phi _{\chi _{\varepsilon }}^{1}
\end{equation*}
is in the normal form, {\em i.e.},
$N_{\varepsilon }\circ R_\alpha=R_\alpha \circ N_{\varepsilon }$.
The normal form can be formally interpolated by an autonomous
Hamiltonian flow:
\begin{equation*}
N_{\varepsilon }=R_{\alpha }\circ \Phi _{H_{\varepsilon }}^{1}
\end{equation*}
where $H_{\varepsilon }$ is a symmetric formal Hamiltonian
\begin{equation*}
H_{\varepsilon }=H_{\varepsilon }\circ R_{\alpha }\,.
\end{equation*}
It is well known that these statements can be proved by an appropriate modification
of the arguments used for the case of individual maps. The leading order
of the normal form has the form
\[
H_\varepsilon(I,\varphi)=I A(I,\varepsilon)+I^{n/2}B(I,\varepsilon)\cos n\varphi+O(I^{n/2+1})\,.
\]
It is also well known the changes of variables involved in the
construction of the normal forms are not unique.
This freedom can used to provide further simplifications
of the normal form Hamiltonian $H_\varepsilon$.

\begin{theorem}\label{Thm:qr}
If $F_\varepsilon$ is a smooth ($C^\infty$ or analytic) family of area preserving maps
such that $F_0$ has a resonant elliptic fixed point at the origin,
then there is a formal Hamiltonian $H_\varepsilon$ and
formal canonical change of variables which conjugates $F_\varepsilon$ with
$R_\alpha\circ\Phi^1_{H_\varepsilon}$.
Moreover, $H_\varepsilon$ has the following form:
\begin{equation}\label{Heps}
H_\varepsilon(I,\varphi)=I A(I,\varepsilon)+I^{n/2}B(I,\varepsilon)\cos n\varphi
\end{equation}
where $H_\varepsilon$ for $\varepsilon=0$
coincides with $H$ of Theorem~\ref{Thm:main} and
\begin{itemize}
\item
if $n\ge4$ and $\partial_I A(0,0)\cdot B(0,0)\ne0$
\begin{equation}
A(I,\varepsilon)=\sum_{k,m\ge0}a_{k,m}I^k\varepsilon ^m\,,\qquad
B(I,\varepsilon)=\sum_{k,m\ge0}b_{k,m}I^{2k}\varepsilon ^m\,, \qquad a_{0,0}=0,
\end{equation}
\item
if $n=3$ and $B(0,0)\ne0$
\begin{eqnarray}
A(I,\varepsilon)&=&\sum_{\substack{k,m\ge0\\ k\ne 1\pmod 3}}a_{k,m}I^k\varepsilon^m\,,\qquad
a_{0,0}=a_{1,0}=0\,,\\
B(I,\varepsilon)&=&\sum_{\substack{k,m\ge0\\ k\ne 2\pmod 3}}b_{k,m}I^k\varepsilon^m\,.
\end{eqnarray}
\end{itemize}
The coefficients of the series $A$ and $B$ are defined uniquely by the map $F_\varepsilon$
provided the leading order is normalized to ensure $b_{00}>0$.
%(In all cases $a_{1,0}=0$.)
\end{theorem}

\subsection*{Structure of the paper}
The rest of the paper is structured in the following way.
In Section~\ref{Se:prep} we provide some useful definitions,
describe the usage of $(z,\bar z)$ variables and
derive several useful formulae.
In Section~\ref{Se:fi} we prove that a
tangent to identity area-preserving map can
be formally interpolated by an autonomous Hamiltonian.
This result is used in the proof of the main theorem
and is included for completeness of the arguments.
Finally, in Sections~\ref{Se:n5}, \ref{Se:n4} and~\ref{Se:n3}
we derive the unique normal forms for the cases $n\ge5$,
$n=4$ and $n=3$ respectively.

Finally, in Section~\ref{Se:quasires} we construct unique normal forms
for the families of area-preserving maps.

\section{Lie series in the complex form\label{Se:prep}}

\subsection{Complex variables}

It is well known that the normal form
theory can look much simpler in the complex coordinates
$z$ and $\bar z$ defined by%\footnote{Who introduced this trick?}
\begin{equation}\label{Eq:zbarz}
z=x+iy\qquad\mbox{and}\qquad
\bar z=x-iy.
\end{equation}
The change $(x,y)\mapsto(z,\bar z)$ is a linear automorphism
of $\mathbb{C}^2$ with the inverse transformation given by
\begin{equation}
x=\frac{z+\bar{z}}{2}
\qquad\text{and}\qquad
y=\frac{z-\bar{z}}{2i}\,.
\label{Eq:complexvar}
\end{equation}
If $x$ and $y$ are both real
then $\bar z$ is the complex conjugate of $z$, i.e., $\bar z=z^*$.

It is important to note that the change is not symplectic
since
\begin{equation}
dx\wedge dy=-\frac{1}{2i}dz\wedge d\bar z\,.
\label{Symplectic_form}
\end{equation}
Nevertheless since the Jacobian is constant,
an area-preserving map also preserves the form $dz\wedge d\bar z$.

Let $(f,\bar f)$ be the components
of a map $F$ in the coordinates $(z,\bar z)$.
We say that $F$ has {\em real symmetry} if
\begin{equation}\label{Eq:realsymm}
\bar f(z,\bar z)=\bigl(f(\bar z^*,z^*)\bigr)^*\,.
\end{equation}
Then the second component of $F$ can be restored
using this symmetry and we do not have to consider
it separately.

We note that $F$ commutes with the involution
$ (z,\bar z)\mapsto(\bar z^*,z^*) $. In the original coordinates
this involution takes the form $(x,y)\mapsto (x^*,y^*)$
and, consequently, $F$ takes real values when both
$x$ and $y$ are real.

For a real-analytic map $F$
the real symmetry described by (\ref{Eq:realsymm})
can be easily restated in terms of its Taylor coefficients.
Then the real symmetry naturally extends from real-analytic
functions onto formal power series and suggests the following definition.

\begin{definition}
We say that the vector-valued series $(f(z,\bar z),\bar f(z,\bar z))$
where $f$ and $\bar f$ are formal series of the form
\[
f(z,\bar z)=\sum_{k,l\ge0}f_{kl}z^k\bar z^l\qquad\text{and}\qquad
\bar f(z,\bar z)=\sum_{k,l\ge0}\bar f_{kl}z^k\bar z^l\,,
\]
has\/ {\em the real symmetry\/} if
\begin{equation}\label{Eq:realsymm_kl}
\bar f_{kl}=f_{lk}^*
\end{equation}
for all $k,l\ge0$.
\end{definition}

This trick substantially simplifies manipulations with power series.
As an example consider the
rotation $(x,y)\mapsto R_{\alpha }(x,y)$ defined by equation~(\ref{Eq:rotation}).
In the complex notation this map takes the diagonal form
$(z,\bar z)\mapsto (\mu z,\mu^*\bar z)$. Indeed, for the first component we get
\begin{eqnarray*}
z=x+iy\ \mapsto\ z_1&=&\left( x\cos \alpha -y\sin \alpha \right) +i\left( x\sin \alpha
+y\cos \alpha \right)  \\
&=&\left( x+iy\right) \left( \cos \alpha +i\sin \alpha \right)
=e^{i\alpha }z=\mu z\,.
\end{eqnarray*}
The formula for the $\bar z$ component of the map is obtained by the real symmetry.

\medskip

We will also consider scalar functions and scalar formal series rewritten
in terms of the variables $(z,\bar z)$.
If $h$ is real-analytic in a neighbourhood of the
origin, it can be expanded in Taylor series
$$
h(z,\bar z)=\sum_{k,l\ge0}h_{kl}z^k\bar z^l\,.
$$
Since $h(z,z^*)$ is real for all $z\in\mathbb{C}$
such that the series converges,
the coefficients are symmetric:
\[
h_{kl}=h^*_{lk}
\]
for all $k,l$. This prompts the
following definition.

\begin{definition}
We say that a formal power series
\[
h=\sum_{k,l\ge0}h_{kl}z^k\bar z^l
\]
is real-valued if $h_{kl}=h^*_{lk}$ for all $k,l$.
\end{definition}

\subsection{Divergence-free vector fields in the complex form}

From now on we assume that the $\bar f$ component
is obtained from $f$ using the real symmetry.
It is convenient to introduce the divergence operator by
\begin{equation}
\func{div}f=
\frac{\partial f}{\partial z}+
\frac{\partial \bar{f}}{\partial \bar{z}}\,.
\end{equation}
In the future proofs we will need the following simple fact.

\begin{lemma}\label{Lemma:areapres}
A real map $(z,\bar z)\mapsto(f(z,\bar z),\bar f(z,\bar z))$
is area-preserving if and only if the function $g(z,\bar z):= f(z,\bar z)-z$
satisfies
\begin{equation}
\label{Eq:areapres}
\func{div} g =\left\{ \bar{g},g\right\}
\end{equation}
where $\bar g$ is obtained from $g$ using the real symmetry.
\end{lemma}

\begin{proof}
Taking into account $f(z,\bar z)=z+g(z,\bar z)$ and using the real symmetry to get $\bar f$
 we obtain
\begin{eqnarray*}
df\wedge d\bar{f}
&=&\left( dz+\frac{\partial g}{\partial z}dz+\frac{\partial g}{\partial
\bar{z}}d\bar{z}\right) \wedge \left(d\bar{z}+\frac{\partial \bar{g%
}}{\partial \bar{z}}d\bar{z}+\frac{\partial \bar{g}}{\partial z}dz\right)  \\
&=&\left( 1+ \frac{\partial \bar{g}}{\partial \bar{z}}+\frac{%
\partial g}{\partial z}+\frac{\partial g}{\partial z}\frac{\partial \bar{g}}{%
\partial \bar{z}}-\frac{\partial g}{\partial \bar{z}}\frac{\partial \bar{g}}{%
\partial z}\right) dz\wedge d\bar{z}.
\end{eqnarray*}
Since $f$ is area-preserving $d f\wedge d\bar{f}=dz\wedge d\bar{z}$ and we get the identity
\[
\frac{\partial g}{\partial z}+
\frac{\partial \bar{g}}{\partial \bar{z}}
=
\frac{\partial g}{\partial \bar{z}}\frac{\partial \bar{g}}{\partial z}-
\frac{\partial g}{\partial z}\frac{\partial \bar{g}}{\partial \bar{z}}
\]
which is equivalent to (\ref{Eq:areapres}).
\end{proof}

Let us consider the Hamiltonian equations
\begin{eqnarray*}
\dot{z} &=&-2i\frac{\partial h}{\partial \bar{z}}, \\
\dot{\bar{z}} &=&2i\frac{\partial h }{\partial z}.
\end{eqnarray*}
Obviously this vector field has zero divergence. Let us consider
a vector field
\begin{equation*}
\dot{z}=g(z,\bar{z})\qquad
\dot{\bar z}=\bar g(z,\bar{z})\,,
\end{equation*}
where $\bar g$ is obtained from $g$ using the real symmetry.
A natural question arises: Suppose $g$ is divergence free,
is it Hamiltonian with a real-valued Hamiltonian function?

The next Lemma gives a positive answer for polynomial
(and consequently for all formal) vector fields.

\begin{lemma}
\label{Lemma_real}
Let $g_{p}$ be a homogeneous polynomial of order $p\ge0$.
There is a real-valued homogeneous polynomial
$h_{p+1}$ of order $p+1$ such that
\begin{equation}
g_{p}=-2i\frac{\partial h_{p+1}}{\partial \bar{z}},  \label{Eq:hamvf}
\end{equation}
if and only if
\begin{equation}\label{Eq:real0div}
\func{div} g_{p} =0\,.
\end{equation}
If exists, the polynomial $h_{p+1}$ is unique
in the class of real-valued homogeneous polynomials
of $z$ and $\bar z$.
Moreover, if for some $\mu$ with $|\mu|=1$ we have $g_{p}(\mu z,\mu^* \bar z)=\mu g_{p}(z,\bar z)$
then $h_{p+1}(\mu z,\mu^*\bar z)=h_{p+1}(z,\bar z)$.
\end{lemma}

\begin{proof}
Suppose
\[
g_p(z,\bar z)=\sum_{k+l=p}a_{kl}z^k\bar z^l
\]
is divergence free. Then
\[
0=\func{div}g_p=\frac{\partial g_p}{\partial z}+
\frac{\partial \bar{g}_p}{\partial \bar{z}}=
\sum_{k+l=p} ka_{kl}z^{k-1}\bar z^l+l a^*_{lk}z^{k}\bar z^{l-1}
\]
where we used the real symmetry to get $\bar g_p$.
Collecting the coefficients in front of $z^k\bar z^l$
we see that $\func{div}g_p=0$ if and only if
\begin{equation}\label{Eq:div0coeff}
(k+1)a_{k+1,l}+(l+1)a^*_{l+1,k}=0
\end{equation}
for all $k,l\ge0$. These relations
involve all coefficients excepting $a_{0p}$.
A homogeneous polynomial of order $p+1$ has the form
\begin{equation}
h_{p+1}=\sum_{k+l=p+1}h_{kl}z^{k}\bar{z}^{l}  \,.
\label{Lesson2_2}
\end{equation}
Substituting this sum into (\ref{Eq:hamvf}),
we easily see that $h_{p+1}$ satisfies the equation
if and only if
\[
h_{kl}=-\frac{a_{k,l-1}}{2 il}
\]
for all $l\ge 1$. These equalities define
all coefficients of $h_{p+1}$ excepting $h_{p+1,0}$.
Equation (\ref{Eq:div0coeff}) implies the real-valuedness
conditions
\begin{equation}
h_{kl}=h^*_{lk}  \label{conjugate_coeff}
\end{equation}
for $l\ge 1$.
The coefficient $h_{p+1,0}$ is not defined by the equation
so we set $h_{p+1,0}=h_{0,p+1}^*$ to extend Equation
(\ref{conjugate_coeff}) onto $l=0$. Then $h_{p+1}$ is real valued.

The other direction of the lemma is trivial since
a Hamiltonian vector field is divergence free.

Finally, if $g$ commutes with the rotation $z\mapsto \mu z$
the Hamiltonian $h_{p+1}$ is invariant with respect to this rotation
due to the explicit formula for its coefficients provided above.
\end{proof}

%\medskip
%
%\begin{example} Let $\chi (x,y)=x^{2}+y^{2}$ be
%a real homogeneous polynomial of order 2. After
%complexifying $\chi $ we obtain $\chi (z,\bar{z})=z\bar{z}$.
%Differentiation with respect to $\bar{z}$ gives
%\begin{equation*}
%\frac{\partial \chi }{\partial z}=\bar{z}\,.
%\end{equation*}
%Hence $\frac{\partial \chi }{\partial z}$ is not real-valued.
%However the mixed second derivative
%$\frac{\partial ^{2}\chi }{\partial z\partial \bar{z}}$
%is again real-valued.
%\end{example}

\subsection{Formal Lie series}

Let $\chi$ and $g$ be two formal power series.
We note that any of the series
involved in next definitions may diverge.
The linear operator defined by the formula
\begin{equation}\label{Def_L}
L_{\chi }g=-2i\left\{g, \chi\right\}_{z,\bar z}
\end{equation}
is called {\em the Lie derivative\/} generated by $\chi$.
We note that if $\chi$ starts with order $p$ and $g$
starts with order $q$, then the series $L_{\chi }g$
starts with order $p+q-2$ as the Poisson bracket
involves differentiation.
We assume $p\ge3$. Then the lowest
order in $L_{\chi }g$ is at least $q+1$ and we
can define the exponent of $L_\chi$ by
\begin{equation}\label{Eq:formalexp}
\exp(L_\chi)g=g+\sum_{k\ge1}\frac1{k!}L_\chi^kg\,,
\end{equation}
where $L_\chi^k$ stands for the operator $L_\chi$
applied $k$ times. The lowest order in the series $L_\chi^kg$
is at least $q+k$ and consequently every coefficient
of $\exp(L_\chi)g$ depends only on a finite number of
coefficients of $\chi$ and $g$. More precisely,
a coefficient of order $n$ depends polynomially
on the coefficients of orders up to $n$.

Let $\mathrm{id}:\mathbb C^2\to\mathbb C^2$ be the identity map
and consider the formal series
\[
\Phi_\chi^1=\exp(L_\chi) \mathrm{id}
\]
where the exponential is applied componentwise.
We note that it is easy to construct the formal series
for the inverse map:
\[
\Phi_\chi^{-1}=\exp(-L_\chi) \mathrm{id}\,.
\]
It follows from the following more general relation:
for any formal series $g$
\begin{equation}\label{Eq:gL}
g\circ \Phi_\chi^1=\exp(L_\chi)g\,.
\end{equation}
Indeed, this equality is known to be valid for convergent series.
In particular, for a polynomial $\chi$ the series
$\Phi_\chi^1$ converges in a neighbourhood of the origin
as it coincides with the Taylor expansion of the time-one
map which shifts a point along trajectories of the Hamiltonian equations:
\begin{eqnarray*}
\dot{z} &=&-2i\frac{\partial \chi }{\partial \bar{z}}, \\
\dot{\bar{z}} &=&2i\frac{\partial \chi }{\partial z}.
\end{eqnarray*}
The factor $2i$ appears due to the symplectic form (\ref{Symplectic_form}).

We can substitute a solution of the Hamiltonian equation
into a function $g(z,\bar z)$.
The Hamiltonian equations and the chain rule imply
\[
\dot g=2i\left\{\chi,g\right\}.
\]
Then repeating the arguments inductively we get a formula for a derivative
of order $n$ in $t$:
\[
g^{(n)}=L_{\chi}^ng.
\]
Writing a Taylor series centred at $0$ for $g\circ \Phi_\chi^t$
and substituting $t=1$ we obtain (\ref{Eq:gL}).
Then the formula extends from polynomials onto formal series
since each order of the series depends only on a finite number
of coefficients.

\section{Formal interpolation\label{Se:fi}}

The next theorem states that a tangent to identity area-preserving
map can be formally interpolated by a Hamiltonian flow.
Note that the theorem says nothing about convergence of the series,
even in the case when the original map is analytic.

\begin{theorem}
\label{Theorem_Interpol}
If a formal series
\[
f(z,\bar z)=z+\sum_{\substack{k+l\ge2\\ k,l\ge0}}
c_{kl}z^k\bar z^l
\]
is a first component of an area-preserving map with
the real symmetry then there exists a unique
real-valued formal Hamiltonian
\[
h(z,\bar z)=\sum_{\substack{k+l\ge3\\ k,l\ge0}}
h_{kl}z^k\bar z^l
\]
such that
\begin{equation*}
f=\exp(L _{h})z.
\end{equation*}
\end{theorem}

\begin{proof}
First we introduce the notation. Let $h_k$, $k\ge3$, be a homogeneous
polynomial of order $k$ and for $m\ge3$
\[
H_m=\sum_{k=3}^m h_{k}(z,\bar{z})\,.
\]
Let $[\,\cdot\,]_k$ denote terms of order $k$ in a formal series.
For example, $[H_m]_k=h_k$ for $3\le k\le m$
and $[H_m]_k=0$ for $k> m$.
Consider the series
\[
\phi^1_{H_m}(z)=\exp(L_{H_m})z\,,
\]
where $z$ stands for the function $(z,\bar z)\mapsto z$.
We will need a more explicit
formula for the exponential map. It is convenient to introduce
\[
L_sg:=2i\left\{h_{s+2},g \right\}
\]
for the Lie derivative generated by the homogeneous polynomial
$h_{s+2}$. If $g$ is a homogeneous polynomial of order $q$, then
$L_s(g)$ is a homogeneous polynomial of order $q+s$.
Therefore $L_s$ increases the order of any homogeneous polynomial by $s$.
Then using the bi-linearity of the Poisson bracket we get
\begin{equation*}
L_{H_m} g =2i\left\{ H_m,g \right\} =2i\sum_{k=3}^m\left\{
h_{k},g \right\} =\sum_{s=1}^{m-2}L_{s}g\,.
\end{equation*}
Substituting this sum into the series for the exponential map we obtain
\begin{eqnarray*}
\phi _{H_m}^{1}(z)&=&z+L_{H_m}z
+\sum_{l\geq 2}\frac{1}{l!}L_{H_m}^lz
\\&=&
z+\sum_{s=1}^{m-2}L_{s}z
+\sum_{l\geq 2}\frac{1}{l!}
\sum_{
1\le s_1,\ldots,s_l\le m-2 }L_{s_1}\cdots L_{s_l}z.
\end{eqnarray*}
Now for every $k\geq 1$ we collect the terms of order $k+1$
\begin{equation}\label{Eq:LSk}
\left[ \phi _{H_m}^{1}(z)\right] _{k+1}
=L_{k}z+\sum_{l=2}^{k}\frac{1}{l!}%
\sum_{\substack{s_{1}+\ldots +s_{l}=k\\
1\le s_1,\ldots,s_l\le m-2}}L_{s_{1}}\ldots L_{s_{l}}z.
\end{equation}
We need one more auxiliary formula. Let us write
$f_p=[f]_p$. We get
\[
f(z,\bar z)=z+\sum_{p\ge2}f_p(z,\bar z).
\]
Substituting this series into equation (\ref{Eq:areapres}) of
Lemma~\ref{Lemma:areapres} and collecting terms of order $p-1$ we get
\begin{equation}\label{eqnCor_RealPart}
\func{div} f_p
=\sum_{k=2}^{p-1}\left\{
\bar f_{k},f_{p-k+1}\right\} .
\end{equation}

\medskip

Our aim is to construct an infinite sequence of $h_k$ such that for every $m\ge3$
\begin{equation}\label{Eq:indFI}
\left[\phi^1_{H_m}\right]_{k}=[f]_k
\qquad
\mbox{for $k=2,\ldots,m-1$.}
\end{equation}
We use the induction. First consider $m=3$. Then equation (\ref{Eq:indFI})
reads
\begin{equation*}
L_{1}z=f_{2}
\end{equation*}
which is equivalent to
\begin{equation}\label{Eq:h3}
 -2i\frac{\partial
h_{3}}{\partial \bar{z}}=f_{2}.
\end{equation}
We note that the sum in the right-hand side of equation
(\ref{eqnCor_RealPart}) with $p=2$  has no terms and consequently $\func{div} f_2=0$.
Then Lemma~\ref{Lemma_real}
implies that there exists a unique real-valued $h_{3}$
which satisfies equation~(\ref{Eq:h3}). This choice of $h_3$
guaranties that (\ref{Eq:indFI}) is satisfied for $m=3$.

\medskip

We start the induction step. Suppose for some $m\ge3$
we have found $h_{3},\ldots ,h_{m}$ such that (\ref{Eq:indFI})
holds. Now we look for
$h_{m+1}$ such that (\ref{Eq:indFI}) holds with $m$
replaced by $m+1$.
Equation (\ref{Eq:LSk}) and the induction assumption
imply that
\[
[\phi^1_{H_{m+1}}]_{k}=[\phi^1_{H_{m}}]_{k}=f_k
\qquad\mbox{for $k\le m-1$.}
\]
Then equation (\ref{Eq:LSk}) implies that the equality $[\phi^1_{H_{m+1}}]_{m}=f_{m}$
is equivalent to
\begin{equation}\label{Eq:hm1}
-2i\frac{\partial h_{m+1}}{\partial \bar z}
+\sum_{l=2}^{m-1}\frac{1}{l!}%
\sum_{\substack{s_{1}+\ldots +s_{l}=m-1\\
1\le s_1,\ldots,s_l\le m-2}}L_{s_{1}}\ldots L_{s_{l}}z
=f_{m}.
\end{equation}
We note that this formula includes $L_s=L_{h_{s+2}}$ with $s\le m-2$
which depend on $h_k$ with $k\le m$.
Therefore we consider (\ref{Eq:hm1}) as an equation for $h_{m+1}$.
In order to show that the equation has a real-valued solution
 we need to check the assumptions of Lemma~\ref{Lemma_real}
are satisfied. Since $\func{div}\left(2i\frac{\partial h_{m+1}}{\partial \bar z}\right)=0$
it is sufficient to check that $\func{div} [\phi^1_{H_{m+1}}]_{m}=\func{div} f_{m}$.
The last property follows from the area preservation. Indeed consider two
area-preserving maps
\begin{equation*}
f(z,\bar{z})=z+\sum_{k\geq 2}{f}_{k}
\qquad\mbox{and}\qquad
\widetilde{f}(z,\bar{z})=z+\sum_{k\geq 2}\widetilde{f}_{k}
\end{equation*}
such that $f_{j}=\widetilde{f}_{j}$ for $2\leq j\leq m-1$. Then
(\ref{eqnCor_RealPart}) implies
\begin{equation*}
\mathrm{div\,}  f_{m}=
\sum_{j=2}^{m-1}\left\{\bar f_{j},f_{k-j}\right\}
=
\sum_{j=2}^{m-1}\left\{ \bar{\tilde f}_{j},\tilde f_{k-j}\right\}
=\mathrm{div\,} \tilde f_{m}.
\end{equation*}
We apply this argument with $\tilde f=\phi^1_{H_{m+1}}$
which is area-preserving by Liouville's theorem. Therefore
equation (\ref{Eq:hm1}) can be uniquely solved with respect to $h_{m+1}$.
The induction step is complete and we have uniquely defined
the desired formal Hamiltonian $h=\sum_{p\ge3}h_p$.
\end{proof}

\begin{remark}
If the map $f(z,\bar z)=\mu z+\sum_{k\ge2}f_k(z,\bar z)$ is in Birkhoff normal form then
its coefficients commute with the rotation $(z,\bar z)\mapsto(\mu z,\mu^*\bar z)$,
{\em i.e.,}  $f_k(\mu z,\mu^* \bar z)=\mu f_k(z, \bar z)$.
The map
\begin{equation*}
\mu^*
f(z,\bar{z})=
z+\sum_{k\ge2}\mu^*f_{k}(z,\bar{z} )
\end{equation*}
satisfies the assumptions of the formal interpolation theorem
which implies that there exists a unique formal Hamiltonian
vector field such that $f=\mu \phi _{h}^{1}$.
Moreover a more accurate analysis of the proof shows that
the Hamiltonian $h$ is in Birkhoff normal form itself.
In other words, the formal Hamiltonian is invariant with
respect to the rotation:
$$
h(\mu z,\mu^* \bar{z})= h(z,\bar{z}).
$$
We remind that the Hamiltonian $h$ is a formal series, and this
relation is to be interpreted termwise.
\end{remark}

\begin{remark}
We also proved that the map $F=(f,\bar f)$ can be approximated by an integrable
map. Indeed, expanding $F-\Phi^1_{H_m}$ into Taylor series
we use equation\/  {\rm (\ref{Eq:indFI})} to show that
the first $m-1$ orders of the series vanish. Then
the standard estimate for a remainder of the Taylor formula
implies that for any $m\ge3$
\[
F=\Phi^1_{H_m}+O(r^{m})
\]
where $r=|z|+|\bar z|$.
\end{remark}
\section{Weak resonances\label{Se:n5}}

Suppose that an area-preserving map $f$ is in the resonant normal form,
{\em i.e.}, $f(\mu z,\mu^*\bar z)=\mu f(z,\bar z)$.
Then there is a formal Hamiltonian such that
$$
f=\mu \phi^1_h=\mu\exp(L_h)z
$$
and $h$ has the symmetry induced by the linear part of $f$:
$$
h(\mu z,\mu^*\bar z)= h(z,\bar z)\,.
$$
The symmetry of the interpolating Hamiltonian $h$
implies that it is represented by a formal series
which contains resonant terms only:
$$
h(z,\bar{z})=\sum_{\substack{ k+l\geq 3 \\ k=l \pmod n}}h_{kl}z^{k}\bar{z}^{l}.
$$
It is easy to see that these series involve a fourth order term
$h_{22}z^2\bar z^2$ independently of $n$. If $n\ge 5$ there are no
other resonant term of an order 4 or less. Then the leading
order is of the same form as in the non-resonant case.
For this reason the resonances with $n\ge5$ are called {\em weak\/}.

Let us consider the case $n\geq 5$ with more details:
\begin{equation}\label{Eq:h_n=5}
h(z,\bar{z})=h_{22}z^{2}\bar{z}^{2}+
\sum_{\substack{ k+l\geq 5 \\ k=l\pmod n}}h_{kl}z^{k}\bar{z}^{l}.
\end{equation}
We will simplify these series using canonical substitutions.

It is convenient to group together terms of the same $\delta$-order.
For a monomial we define its $\delta$-order by
\begin{equation}\label{Eq:deltamon}
\delta(z^k\bar z^l)=
2\left|\frac{k-l}n\right|+\min\{\,k,l\,\}
=
\frac12(k+l)-\frac{n-4}{2n}|k-l|\,.
%\left\{
%\begin{array}{ll}
%\displaystyle
%\frac{2k}{n}+\left(1-\frac{2k}{n}\right)l,&k\ge l\\
%\displaystyle
%\frac{2l}{n}+\left(1-\frac{2k}{n}\right)k,&k\le l
%\end{array}
%\right.
\end{equation}
Let $\mathcal H^n_m$ denote the set of all real-valued polynomials
which can be represented as a sum of resonant monomials
of the $\delta$-order $m$. For example, $\mathcal H^n_2$ consists of
polynomials of the form $c_{n0} z^n+c_{22} z^2\bar z^2+c_{0n}\bar z^n$
with $c_{0n}=c_{n0}^*\in\mathbb C$ and $c_{22}\in\mathbb R$.
Therefore $\mathcal H^n_2$ is a three dimensional real vector space.
The resonant terms are sketched on Figure~\ref{Fig:resn}.
\begin{figure}
\begin{center}
\includegraphics[width=0.60\textwidth]{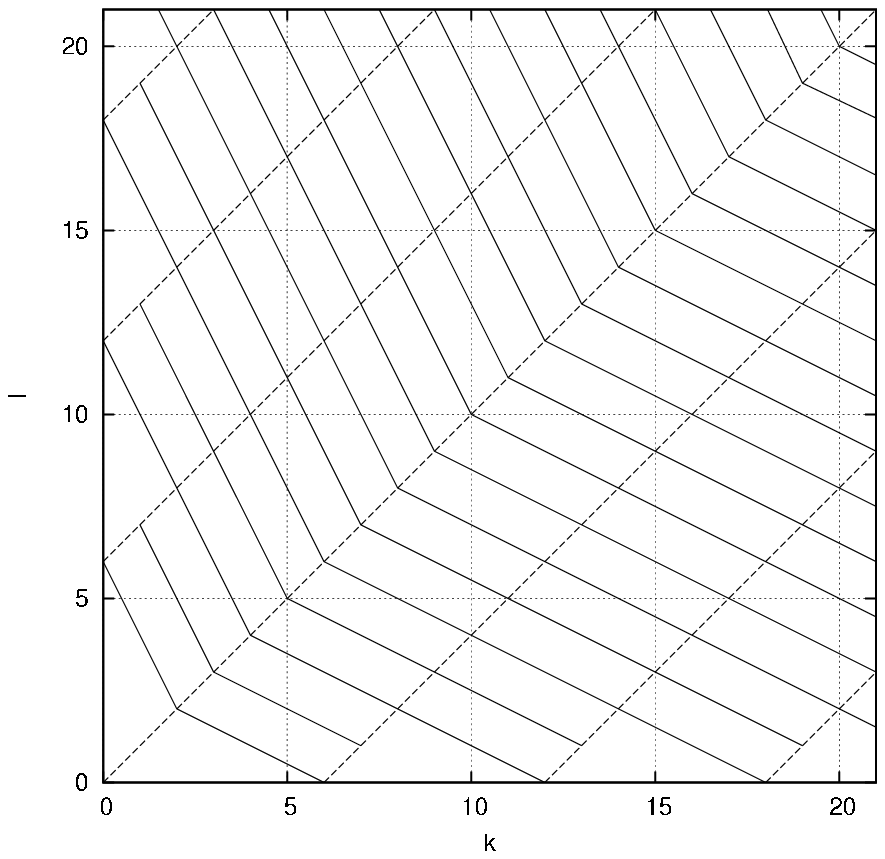}
\end{center}
\caption{Resonant terms for $n=6$.
On the diagram the resonant terms correspond to
intersections of solid and dashed lines on the $(k,l)$ plane. Each of the solid lines
connects the points of equal $\delta$-order.\label{Fig:resn}}
\end{figure}

Let $\ipart{\cdot}$ denote the integer part of a number. Assume $m\ge1$.

\begin{lemma}\label{Le:dimhp}
The set $\mathcal H^n_m$ is a
real vector space of dimension $1+2\ipart{\frac m2}$\,.
\end{lemma}
\begin{proof}
If a resonant monomial $z^k\bar z^l$ has the $\delta$-order $m$ then
\begin{eqnarray*}
k&=&l+nj\qquad\mbox{for some $j\in\mathbb Z$ (resonant term),}\\
m&=&2|j|+\min\{\,k,l\,\}\qquad \mbox{($\delta$-order equals $m$).}
\end{eqnarray*}
Let us count the number of monomials which satisfy these two conditions.
There is one with $j=0$: $k=l=m$.
Then there are $\ipart{\frac m2}$ monomials
with~$j>0$. Indeed, since $k>l$ we get
\begin{eqnarray*}
l&=&m-2 j\qquad 1\le j\le \ipart{\frac m2}\,,\\
k&=&l+n j\,.
\end{eqnarray*}
We also get the equal number of monomials with $j<0$ due to the symmetry.
It is convenient to denote the resonant monomials by
\begin{equation}\label{Eq:Pmj}
Q_{m,j}=z^{m+nj-2j}\bar z^{m-2j}\qquad\mbox{and}\qquad
Q_{m,-j}=z^{m-2j}\bar z^{m+nj-2j}
\end{equation}
for $0\le j\le \ipart{\frac m2}$.
Then any resonant polynomial which contains only monomials of the $\delta$-order $m$
has the form $\sum_{j=-\ipart{\frac m2}}^{\ipart{\frac m2}}c_{j}Q_{mj}$.
Taking into account that  $c_{kl}=c_{lk}^*$ due to real-valuedness,
we conclude that the real dimension of the space $\dim\mathcal H^n_m=1+2\ipart{\frac m2}$.
\end{proof}

\begin{lemma}\label{Le:ordershift}
Let $n\ge 4$. If $z^{k_1}\bar z^{l_1}$ and $z^{k_2}\bar z^{l_2}$ are two resonant monomials of
$\delta$-orders $m_1$ and $m_2$ respectively, then
\[
\{z^{k_1}\bar z^{l_1},z^{k_2}\bar z^{l_2}\}=(k_1l_2-k_2l_1)z^{k_1+k_2-1}\bar z^{l_1+l_2-1}
\]
is a resonant monomial of $\delta$-order $m\ge m_1+m_2-1$.
\end{lemma}
\begin{proof}
The Poisson bracket of the monomials has the form $\mathrm{const}\,z^k\bar z^l$ with $k=k_1+k_2-1$
and $l=l_1+l_2-1$. Using the second formula from the definition of the $\delta$ order (\ref{Eq:deltamon})
we get
\begin{eqnarray*}
m&=&\frac12(k+l)-\frac{n-4}{2n}|k-l|\\
&=&
\frac12(k_1+k_2+l_1+l_2-2)-\frac{n-4}{2n}|k_1+k_2-l_1-l_2|\,.
\end{eqnarray*}
Then we rewrite it in the form
\begin{eqnarray*}
m&=&
\frac12(k_1+l_1)-\frac{n-4}{2n}|k_1-l_1|+\frac12(k_2+l_2)-\frac{n-4}{2n}|k_2-l_2|-1\\
&&\qquad +
\frac{n-4}{2n}\left(
|k_1-l_1| +|k_2-l_2|-
|k_1-l_1+k_2-l_2|
\right)\,.
\end{eqnarray*}
Since the last parenthesis is not negative and $n\ge4$ we conclude
\begin{eqnarray*}
m&\ge&
\frac12(k_1+l_1)-\frac{n-4}{2n}|k_1-l_1|+\frac12(k_2+l_2)-\frac{n-4}{2n}|k_2-l_2|-1\\
&=&m_1+m_2-1
\end{eqnarray*}
which completes the proof of the lemma.
\end{proof}

Now we state a proposition which is the central part of our main theorem.

\begin{proposition}\label{Pro:n5}
If $n\ge5$ and $h_{22},h_{n0}\neq 0$ there exists a formal canonical change of variables which
transforms a formal real-valued Hamiltonian\/ {\rm (\ref{Eq:h_n=5})}
into
\begin{equation}\label{Eq:cnfforn}
\tilde h:=(z\bar{z})^{2}A(z\bar{z})+(z^{n}+\bar z^n)B(z^2\bar z^2),
\end{equation}%
where $A,B\in \mathbb{R}[[z\bar{z}]]$ (formal series
with real coefficients in the single variable $z\bar z$)
and $A(0)=h_{22}$, $B(0)=|h_{n0}|$.
Moreover, the coefficients of the series $A$ and $B$ are defined uniquely.
\end{proposition}

\begin{proof}
The proposition is proved by induction. We perform a sequence of canonical
coordinate changes normalising one  $\delta$-order of the formal
Hamiltonian at a time.

Let us write $[h]_p$ to denote the terms of the $\delta$-order $p$
in the formal series $h$. In particular,
$[h]_2=h_{n0}z^n+h_{22}z^2\bar z^2+h_{0n}\bar z^n$.
The rotation $z\mapsto z\exp(-i\arg(h_{n0})/n )$ transforms
it into
\begin{equation}\label{Eq:h2}
h_2:=[h]_2=b_{0}z^n+a_{0}z^2\bar z^2+b_{0}\bar z^n
\end{equation}
where $b_0=|h_{n0}|$ and $a_0=h_{22}$ are both real and positive.
We keep the same letter $h$
for the transformed Hamiltonian hoping that it will cause no confusion.
After the substitution the leading term of $h$ has the desired form.

All other substitutions are constructed using Lie series (see e.g. \cite{Deprit1969}).
Take a polynomial $\chi_p\in\mathcal H^n_p$, $p\ge2$, and make a substitution
generated by $\chi_p$. By (\ref{Eq:gL}), the new Hamiltonian takes the form
\[
\tilde h=h+L_{\chi_p}h+\sum_{k\ge2}\frac1{k!}L^k_{\chi_p}h\,.
\]
Lemma~\ref{Le:ordershift} implies that the series
$L^k_{\chi_p}h$ starts with the $\delta$-order $k(p-1)+2$
or higher. Therefore each term of $\tilde h$ depends only
on a finite number of terms in the series $h$. Moreover,
for $2\le m\le p$ we get
\[
[\tilde h]_m=[h]_m
\]
and
\begin{equation}\label{Eq:homol}
[\tilde h]_{p+1}=[h]_{p+1}+\left[L_{\chi_p}h_2\right]_{p+1}\,.
\end{equation}
We choose $\chi_p$ to transform this $\delta$-order to
the desired form. For this purpose let us consider
the linear operator $L_p:\mathcal H^n_p\to \mathcal H^n_{p+1}$
defined by
\[
L_p\chi_p= \left[L_{\chi_p}h_2\right]_{p+1}\,.
\]
It is sometimes called the {\em homological operator}.
We will find a subspace complement to $L_p(\mathcal H^n_p)$ in $\mathcal H^n_{p+1}$
and choose $\chi_p$ to ensure that $[\tilde h]_{p+1}$ belongs to this
subspace. The properties of $L_p$ are
slightly different for odd and even values of $p$. Let us state these properties first.

If $p=2k+1$ is odd the kernel of $L_p$ is trivial.
Lemma~\ref{Le:dimhp} implies that $\dim\mathcal H^n_{p+1}=\dim\mathcal H^n_p+2$.
Therefore
\[
\mbox{co-dim}\,\mathrm{Image}(L_{2k+1})=2\,.
\]

If $p=2k$ is even, then
$\dim\mathcal H^n_{p+1}=\dim\mathcal H^n_p$.
The kernel of $L_p$ is one-dimensional:
it is generated by multiples of $h_2^k$. Therefore
\[
\mbox{co-dim}\,\mathrm{Image}(L_{2k})=1\,.
\]
In order to prove these claims
and provide an explicit description
for the complements,
we find a matrix which describes $L_p$.
We note that any polynomial from $\mathcal H^n_p$ can
be written in the form
\[
\chi_p=c_{0}Q_{p,0}+\sum_{j=1}^{\ipart{\frac p2}}
(c_{j}Q_{p,j}+c_{j}^* Q_{p,-j})
\]
where $c_0$ is real and $c_j$ with $j\ge1$ may be complex.
The monomials $Q_{p,j}$ are defined by (\ref{Eq:Pmj}).
Using (\ref{Eq:h2}) we compute the action of $L_p$ on monomials:
\[
L_p(z^p\bar z^p)=-2i b_0 np\left( z^{n+p-1}\bar z^{p-1}-z^{p-1}\bar z^{n+p-1}
\right)
\]
and
\begin{eqnarray*}
L_p(z^{p-2j+nj}\bar z^{p-2j})&=&
4i a_0 n j \, z^{p-2j+nj +1}\bar z^{p-2j+1}\\&&
-2i b_0 n(p-2j) \, z^{p-2j+nj+n-1}\bar z^{p-2j-1}\,.
\end{eqnarray*}
These formulae can be rewritten:
\begin{eqnarray*}
L_p(Q_{p,0})&=&-2i b_0 np Q_{p+1,1}+ 2i b_0 np Q_{p+1,-1} \,,\\
L_p(Q_{p,j})&=&
4i a_0 n j \, Q_{p+1,j}%z^{p-2j+nj +1}\bar z^{p-2j+1}
-2i b_0 n(p-2j) \, Q_{p+1,j+1}% z^{p-2j+nj+n-1}\bar z^{p-2j-1}
\,,
\qquad
1\le j\le\ipart{\frac{p}2}\,.
\end{eqnarray*}
Since $L_p\chi_p\in\mathcal H^n_{p+1}$ we can represent it in the form
\[
L_p\chi_p
=
d_{0}Q_{p+1,0}+\sum_{j=1}^{\ipart{\frac {p+1}2}}
(d_{j}Q_{p+1,j}
+d_{j}^*Q_{p+1,-j}
)
\]
for some constant $d_j$, $0\le j\le \ipart{\frac{p+1}2}$. From the explicit
formulae we see that the image of $L_p$ does not contain terms
proportional to $Q_{p+1,0}=z^{p+1}\bar z^{p+1}$. Therefore the complement to the image
is at least one dimensional. In the image we get $d_0=0$ and
\begin{equation}\label{Eq:dj}
d_j=4i a_0nj\,c_j-2ib_0n(p-2j+2)\,c_{j-1}
\qquad\mbox{for}\quad
1\le j\le \ipart{\frac{p}2}\,.
\end{equation}
If $p=2k+1$ is odd, there is an additional equality:
\[
d_{k+1}=-2ib_0n\,c_{k}\,.
\]
In this case the map $L_p$ considered as
an operator  which maps $(c_0,\dots,c_k)\mapsto(d_1,\dots,d_{k+1})$
is a linear isomorphism of $\mathbb C^{k+1}$. Indeed,
the corresponding matrix is triangle and its
determinant equals to the product
of the diagonal elements: $(-2ib_0n)^{k+1}(2k)!!$.
In this representation the space
 $\mathcal H^n_p$ of real-valued
Hamiltonians  is identified with
the subspace $\{\,\Im c_0=0\,\}$ of $\mathbb C^{k+1}$.
The operator $L_p$ maps
the vector $(i ,0,\dots,0)$ into $(2b_0n p ,0,\ldots,0)$.
We see that the preimage of the real-valued polynomial
$Q_{p+1,1}+Q_{p+1,-1}$ is not real-valued.
Therefore the complement to $L_{2k+1}(\mathcal H^n_{2k+1})\subset \mathcal H^n_{2k+2}$
is two dimensional and consists of polynomials of the form
\begin{equation}\label{Eq:podd}
d_0 Q_{p+1,0 }
+
d_1 (Q_{p+1,1}+Q_{p+1,-1})
\end{equation}
with $d_0,d_1\in\mathbb R$.

Now consider the case of $p=2k$. First we restrict the operator $L_{2k}$
onto the vectors with $c_0=0$ and note that
$L_{2k}:(0,c_1,\ldots,c_k)\mapsto(0,d_1,\ldots,d_k)$.
Equation (\ref{Eq:dj}) implies that this map is
a linear isomorphism (of $\mathbb C^k$).
Indeed, the corresponding matrix is triangle
and its determinant is the product of its diagonal elements: $ (4i a_0n)^k\,k!\ne0$
and the matrix is invertible. Therefore the complement
to $L_{2k}(\mathcal H^n_{2k})\subset \mathcal H^n_{2k+1}$
is one dimensional and consists of monomials of the form
\begin{equation}\label{Eq:peven}
d_0 Q_{p+1,0}
\end{equation}
where $d_0\in\mathbb R$ due to real-valuedness.

We conclude that in the homological equation (\ref{Eq:homol})
the auxiliary polynomial $\chi_p$ can be chosen in
such a way that $[\tilde h]_{p+1}$ is either of the form
(\ref{Eq:podd}) or~(\ref{Eq:peven}).

We continue inductively
starting with the $\delta$-order $3$.
We note that the substitution
$\Phi _{\chi _{p}}^{1}$
does not change $\delta$-orders $k\leq p$ and
the composition of the changes is a well-defined formal series.
Therefore the original Hamiltonian $h$ can be transformed
in such a way that each order is either of the form
(\ref{Eq:podd}) or~(\ref{Eq:peven}).
Taking into account the definition of $Q_{p,j}$ we see that
$h$ is transformed to the desired form (\ref{Eq:cnfforn}).

\medskip

In order to complete the proof we need to establish uniqueness of
the series (\ref{Eq:cnfforn}). We note that the transformation
constructed in the first part of the proof is not unique
because the kernel of $L_{2k}$ is not empty.
Nevertheless the normalised Hamiltonian is unique.
Indeed, suppose that two Hamiltonians of the form
(\ref{Eq:cnfforn}) are conjugate, {\em i.e.}, there is a formal
Hamiltonian $\chi$ such that
\begin{equation}\label{Eq:subst}
\tilde h'=\exp(L_\chi)\tilde h\,.
\end{equation}
Let $p$ be the lowest $\delta$-order of the formal series $\chi$.
Then
\[
[\tilde h']_m=[\tilde h]_m
\]
for $2\le m\le p$ and
\[
[\tilde h']_{p+1}=[\tilde h]_{p+1}+L_p([\chi]_p)\,.
\]
Since both $[\tilde h']_{p+1}$ and $[\tilde h]_{p+1}$ are in the complement
to the image of $L_p$ we conclude that
\([\tilde h']_{p+1}=[\tilde h]_{p+1}\)
and $L_p([\chi]_p)=0$. Therefore $[\chi]_p$
is in the kernel of $L_p$. Since for
all odd $p$ the kernel is trivial, $p$ is even.
For an even $p$ the kernel is one dimensional
and consequently $[\chi]_p=\alpha [h_2^{p/2}]_p$
for some $\alpha\ne0$. Then
obviously $\tilde h=\exp(-\alpha L_{\tilde h^{p/2}})\tilde h$
and we obtain
\[
\tilde h'=\exp(L_\chi)\exp(-\alpha L_{\tilde h^{p/2}})\tilde h\,.
\]
The composition of two tangent to identity maps is also tangent to
identity, and Theorem~\ref{Theorem_Interpol} implies that there is
a formal series $\tilde \chi$ such that
\[
\exp(L_{\tilde\chi})=\exp(L_\chi)\exp(-\alpha L_{\tilde h^{p/2}})\,.
\]
Then $\tilde h'=\exp(L_{\tilde\chi})\tilde h$.
We obtained an equation of the form (\ref{Eq:subst})
but the lowest $\delta$-order of $\tilde\chi$ is at least $p+2$.
Then the argument can be repeated starting with (\ref{Eq:subst})
to show that $\tilde h$ and $\tilde h'$ coincide at all orders.
\end{proof}

\section{Forth order resonance\label{Se:n4}}

In the case $n=4$ the construction of the unique normal form
is similar to the construction used for the case of a weak resonance.
Nevertheless this case is to be considered separately since
the matrix of the homological operator is not  triangle.

The Hamiltonian $h$ is given by
\begin{equation*}
h(z,\bar{z})=
\sum_{\substack{ k+l\geq 4 \\ k=l\pmod 4}}h_{kl}z^{k}\bar{z}^{l}.
\end{equation*}
In the case of $n=4$, the $\delta$-order defined by (\ref{Eq:deltamon})
is just a half of the usual order of a polynomial.

The leading terms of the series are of order 4
and correspond to $(k,l)$ equal to
$(4,0)$, $(2,2)$ and $(0,4)$.
The coefficient $h_{22}$ is real due to the real valuedness.
Without loosing in generality we can assume that
$h_{40}$ is real and positive which can be achieved by rotating
the complex plane using the substitution:
$z\mapsto e^{-i\arg(h_{40})/4 }\,z$.
Then taking into account the real valuedness of $h$
we can write the terms of order four in the following form:
\begin{equation}
h_{2}(z,\bar{z})=a_0z^{2}\bar{z}^{2}+b_{0}(z^4+\bar z^{4}),
\label{n=4_LeadingOrder}
\end{equation}
where $a_0=h_{22}$ and $b_0=|h_{40}|$.

\begin{proposition}\label{Pro:n4}
If $h_{40}\ne0$, there is a formal canonical change
of variables which transforms $h(z,\bar{z})$
into
\begin{equation*}
\widetilde{h}(z,\bar{z})
=z^{2}\bar{z}^{2}A(z\bar{z})+(z^4+\bar z^4)B(z^2\bar{z}^2),
\end{equation*}%
where $A$ and $B$ are series in one variable with real coefficients,
$A(0)=h_{22}$, $B(0)=|h_{40}|$.
Moreover the coefficients of the series $A$ and $B$ are unique.
\end{proposition}

\begin{proof} First we note that the resonant terms correspond
to $k=l \pmod 4$ which is equivalent to
\begin{equation*}
k=l+4j,\quad j\in \mathbb{Z}\,.
\end{equation*}
Then
\begin{equation*}
k+l=2l+4j,\quad j\in \mathbb{Z}
\end{equation*}%
which implies that $k+l$ are all even.
It is convenient to illustrate distribution
of the resonant terms using the diagram shown on Figure~\ref{Fig:n4res}.
\begin{figure}
\begin{center}
\includegraphics{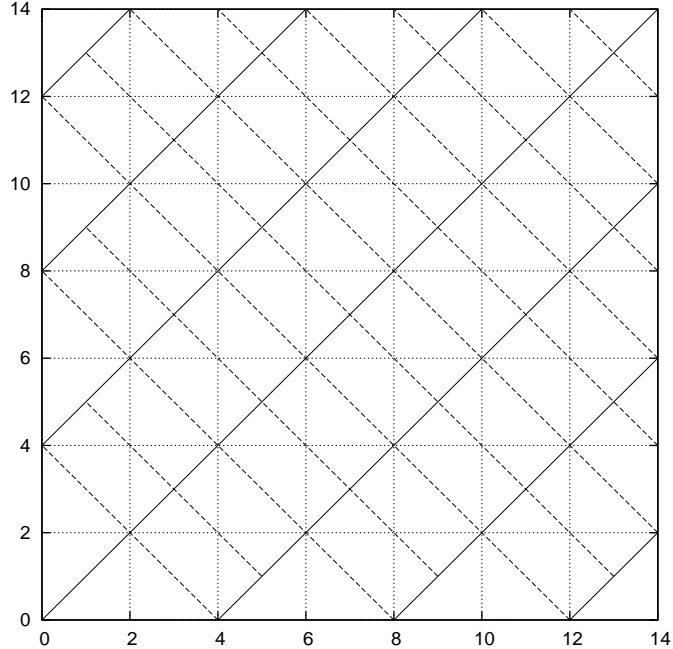}
\end{center}
\caption{On the $(k,l)$ plane, resonant terms in the Hamiltonian for $n=4$
correspond to intersections of solid and dashed lines.
The dashed lines connect the terms of equal $\delta$-order\label{Fig:n4res}}
\end{figure}

We will prove the proposition by induction transforming
the Hamiltonian to the desired form order by order.
On each step we will need to solve a homological equation
which involves the operator $L_2$ defined by
\begin{equation*}
L_{2}(\chi )=2i\left\{ \chi ,h_{2}\right\} .
\end{equation*}
Before proceeding further
let us study the action of this operator on homogeneous
polynomials.
Let us introduce $\mathcal{H}_{m}^{4}$ as the space of
resonant terms of order $2m$. Unlike the previous section,
where we used the real vector spaces,
it is more convenient to assume that $\mathcal{H}_{m}^{4}$
consists of linear combination of resonant monomials
\[
Q_{m,j}=z^{m+2j}\bar z^{m-2j},\qquad{-\ipart{\frac m2}\le j\le \ipart{\frac m2},}
\]
with complex coefficients. Consequently,
\begin{equation}
\dim \mathcal{H}_{m}^{4}
=2\left\lfloor \frac{m}{2}\right\rfloor +1\,.
\label{n=4_NoOfTerms}
\end{equation}%
Since $\left\lfloor \cdot \right\rfloor $ denotes the integer part of a number, the number of the
resonant monomials of a given order form the sequence $3,3,5,5,7,7,\ldots $ (see
Figure~\ref{Fig:n4res}).

Then real-valued polynomials form a (real) subspace in
$\mathcal{H}_{m}^{4}$ (the coefficients in front of $Q_{m,\pm j}$ are mutually complex
conjugate).

Taking into account (\ref{n=4_LeadingOrder})
we obtain%
%
%\footnote{In particular, the action of $L_{2}$ on a monomial $z^k\bar z^l$ is given by
%\begin{equation*}
%L_{2}(z^{k}\bar{z}^{l})=4ia_{0}(k-l)z^{k+1}\bar{z}^{l+1}+8ib_{0}\left(
%kz^{k-1}\bar{z}^{l+3}-lz^{k+3}\bar{z}^{l-1}\right) .
%\end{equation*}}
%
\begin{eqnarray*}
L_{2}(\chi ) &=&2i\left\{ \chi ,h_{2}\right\}  \\
&=& 4ia_{0}\left( z^{2}\bar{z}\frac{\partial \chi }{\partial z}-z%
\bar{z}^{2}\frac{\partial \chi }{\partial \bar{z}}\right) +8ib_{0}\left(
\bar{z}^{3}\frac{\partial \chi }{\partial z}-z^{3}\frac{\partial \chi }{%
\partial \bar{z}}\right)  .
\end{eqnarray*}
We see that if $\chi$ is a homogeneous polynomial of order $p$
then $L_2(\chi)$ is a homogeneous polynomial of order $p+2$.
Moreover $L_2$ maps resonant monomials into resonant ones.
Then
\begin{equation*}
L_{2}:\mathcal{H}_{m}^{4}\rightarrow \mathcal{H}_{m+1}^{4}.
\end{equation*}%
Let us find a subspace complement to $L_2(\mathcal{H}_{m}^{4})$
in $\mathcal{H}_{m+1}^{4}$.

A straightforward substitution into
the definition of $L_2$ shows that
\begin{eqnarray*}
L_2(Q_{m,j})&=&
16ia_{0}j Q_{m+1,j}
\\&&
+
8ib_{0}
(m+2j)Q_{m+1,j-1}%z^{m+2j-1}\bar{z}^{m+3-2j}
%\\&&
-8ib_{0}(m-2j)Q_{m+1,j+1}%z^{m+2j+3}\bar{z}^{m-2j-1}
\,.
\end{eqnarray*}
Then $L_2$ is described by a tridiagonal
matrix with coefficients $s_{kj}$ given by
\begin{equation}\label{Eq:L2matrix}
s_{j,j+ 1}=- 8ib_{0}(m-2j),\qquad
s_{j,j- 1}=8ib_{0}(m+2j),\qquad
s_{j,j}=16ia_{0}j\,.
\end{equation}
We note that equation $\left( \ref{n=4_NoOfTerms}\right) $ gives us
\begin{equation*}
\dim\mathcal H_{m}^4
=\left\{
\begin{array}{ll}
m\,,&m\text{ odd,} \\
m+1\,,
&
m\text{ even.}
\end{array}%
\right.
\end{equation*}%
So we have to treat two separate cases, namely $m$ odd and $m$ even.

Let us start by considering the case when $m$ is odd. In this case
\begin{equation*}
\dim \mathcal{H}_{m}^{4}=m
\end{equation*}
and
\begin{equation*}
\dim \mathcal{H}_{m+1}^{4}=m+2\,.
\end{equation*}
We see that $L_2$ acts from a space of a lower dimension into
a space of a higher dimension. Its matrix has $(m+2)$ rows and $m$
columns. The non-zero elements are given by (\ref{Eq:L2matrix})
where $-\frac{m-1}2\le j\le\frac{m-1}2$.
For example in the case of $m=5$
the matrix has the following structure
\begin{equation*}
\left[
\begin{array}{ccccc}
x & 0 & 0 & 0 & 0 \\
x & x & 0 & 0 & 0 \\
x & x & x & 0 & 0 \\
0 & x & 0 & x & 0 \\
0 & 0 & x & x & x \\
0 & 0 & 0 & x & x \\
0 & 0 & 0 & 0 & x%
\end{array}%
\right]
\end{equation*}%
where $x$ occupies positions of non-zero elements. It is easy to see that
since $b_0\ne0$ there is
a $m\times m$ block with a non vanishing determinant
(for example the first m rows form a lower diagonal matrix
so its determinant is a straightforward product).
We conclude that if $m$ is odd
\begin{equation*}
\func{rank}\left(L_{2}\right) =m.
\end{equation*}
Since the rank is maximal the kernel is trivial:
\begin{equation*}
\ker (L_{2})=\mathbf{0}.
\end{equation*}
For $m$ even we have
\begin{equation*}
\dim \mathcal{H}_{m}^{4}=
\dim \mathcal{H}_{m+1}^{4}=
m+1\,.
\end{equation*}%
Then $L_2$ is described by a square
$(m+1)\times(m+1)$ tridiagonal matrix.
The non-zero elements are given by (\ref{Eq:L2matrix})
where $-\frac{m}2\le j\le\frac{m}2$.
For example for $m=4$ the matrix takes the form
\begin{equation*}
\left[
\begin{array}{ccccc}
x & x & 0 & 0 & 0 \\
x & x & x & 0 & 0 \\
0 & x & 0 & x & 0 \\
0 & 0 & x & x & x \\
0 & 0 & 0 & x & x%
\end{array}%
\right]
\end{equation*}%
where $x$ occupies places of non-zero elements. To determine the rank of this matrix
we first note that since $b_0\ne0$ the lower left block of size $m\times m$ is
upper-diagonal with a non-zero determinant. Therefore
$
\func{rank}\left(L_{2}\right) \geq m.
$
On the other hand $m$ is even and $h_{2}^{m/2}$ is a
resonant homogeneous polynomial of order $2 m$. It is  in
the kernel of $L_{2}$ because
\begin{equation*}
L_{2}(h_{2}^{m/2})=2i\{h_{2}^{m/2},h_2\}=0\,.
\end{equation*}
Consequently $\func{rank}\left(L_{2}\right) <m+1$.
We conclude that
\begin{equation*}
\func{rank}\left( L_{2}\right) =m
\end{equation*}
and the kernel is one dimensional,
\begin{equation*}
\dim (\ker (L_{2}))=1\,,
\end{equation*}
and consists of elements proportional to $h_{2}^{m/2}$.

Summarising these results we see that
\begin{equation*}
\mbox{co-dim} (L_{2}(\mathcal H_{m}^4)) =
\left\{\begin{array}{ll}2,\qquad &\text{if $m$ is  odd,} \\
1,\qquad &\text{if $m$ is even.}
\end{array}
\right.
\end{equation*}
The real-valued polynomials form a (real) subspace of ${\cal H}^4_{m}$.
The operator $L_2$ preserves real-valuedness.
Consequently, the real co-dimensions of the image of $L_{2}$
restricted on the spaces of real-valued polynomials
are given by the same formula.

Now we need an explicit description for the complements.
A straightforward computation which uses the explicit
formulae for the matrix coefficients (\ref{Eq:L2matrix})
shows if $m$ is odd the polynomials
\[
d_1 z^{m+3}\bar z^{m-1}+d_0z^{m+1}\bar z^{m+1}
+d_1 z^{m-1}\bar z^{m+3}
\]
with real $d_0,d_1$
do not have preimages under $L_2$.
The case of even $m$ is a bit more complicated.
In this case the matrix of $L_2$ is square and
its determinant vanishes. We replace the central column
of this matrix (the one which corresponds to $j=0$)
by the vector $(0,\ldots,0,d_0,0,\ldots,0)^T$
and check that the new matrix has a non-zero
determinant. The computation of the determinant
takes into account that the new
matrix has block structure and each of the blocks
is tridiagonal. Consequently, the added vector is linearly independent
from the columns of the matrix of $L_2$ and belongs
to the complement to its image. Therefore for $m$ even
\[
d_0z^{m+1}\bar z^{m+1}
\]
is not in the image of $L_2$.

We constructed two subspaces of
dimensions two and one respectively
which have trivial intersection with the image of $L_2$.
Consequently, they provide the desired complements
to $L_2(\mathcal H_{m}^4)$.
We note that these complements are described by the same
formulae as in the case $n\ge5$.

\medskip

Now we proceed to the proof of the proposition.
Let $\chi_{m}$ be a real-valued homogeneous
polynomial of order $2m$. Then
\begin{equation*}
\widetilde{h}=h\circ \Phi _{\chi _{m}}^{1}=
\exp(L_{\chi_{m}})h=h+L_{2}(\chi_{m})+O_{2m+4}\,,
\end{equation*}
where $O_{2m+4}$ denotes a formal series without terms of orders
lower than $2m+4$.
We remind that $L_{2}$ increases the order of a homogeneous
polynomial by 2. Therefore
$h$ and $\tilde h$ coincide up to the order $2m+1$ and
\[
[\tilde h]_{m+1}=[h]_{m+1}+L_{2}(\chi_{m})\,.
\]
We choose $\chi_{m}$ in such a way that $[\tilde h]_{m+1}$ is in
the complement to $L_2(\mathcal H_{m}^4)\subset\mathcal H_{m+1}^4$.
Then replace $m$ by $m+1$ and repeat the procedure.

\medskip

The proof of the uniqueness uses essentially the same arguments
as we used in the previous section. Suppose
$h$ can be transformed to two different simplified normal forms
$\tilde h$ and $\tilde h'$ due to non-uniqueness of
transformations to the normal form. Then there is
a canonical transformation $\phi$ such that
\[
\tilde h=\tilde h'\circ\phi\,.
\]
Since the transformation $\phi$ is tangent to identity
there is a formal real-valued Hamiltonian $\chi$ such that
\[
\phi=\Phi^1_{\chi}\,.
\]
Suppose that $2p$ is the lowest order of $\chi$.
Then $\tilde h$ and $\tilde h'$ coincide
up to the order $2p+1$ and
\[
[\tilde h]_{2p+2}=[\tilde h']_{2p+2}+L_2(\chi_{2p})\,.
\]
Since both $[\tilde h]_{2p+2}$ and $[\tilde h']_{2p+2}$
are in the complement subspace to the image of $L_2$,
we conclude that $L_2(\chi_{2p})=0$ and
$[\tilde h]_{2p+2}=[\tilde h']_{2p+2}$.

Moreover either $\chi_{2p}=0$ if $p$ odd, or $\chi_{2p}=c\, h_2^{p/2}$
for some $c\in\mathbb{R}$ if $p$ is even.
Then the change of variables
\[
\tilde \phi= \Phi^1_{\chi} \circ\Phi^1_{-c {\tilde h}^{p/2}}
\]
also transforms $\tilde h'$ into $\tilde h$.
It is easy to check that the corresponding Hamiltonian $\tilde \chi$
starts with order $p+2$.

Repeating the arguments inductively we see that $\tilde h$ and $\tilde h'$
coincide at all orders. Hence the simplified normal form is unique.
\end{proof}

\begin{remark}
In the symplectic polar coordinates $(I,\varphi)$
the normal form
\begin{equation*}
\widetilde{h}(z,\bar{z})=z^{2}\bar{z}^{2}a(z\bar{z})+(z^{4}+\bar z^4)b(z\bar{z}),
\end{equation*}
takes the form
\begin{equation*}
{H}(I,\varphi)=I ^{2}A(I)+I ^{2}B(I )\cos (4\varphi ).
\end{equation*}
We remind that $I=z\bar z/2$, $\varphi=\arg(z)$
or equivalently $z=\sqrt{2I} \,e^{i\varphi}$.
\end{remark}

\section{Third order resonance\label{Se:n3}}

This case is reduced to a resonant Hamiltonian of the form
\begin{equation}\label{Eq:ham3}
h(z,\bar z)=\sum_{\substack{k+l\ge3\\ k=l\pmod 3}}
h_{kl}z^k\bar z^l
\,.
\end{equation}

\begin{proposition}\label{Pro:n3}
If $h_{30}\ne0$, there is a formal canonical change
of variables which transforms $h(z,\bar{z})$
into
\begin{equation}\label{Eq:complexnf3}
\widetilde{h}(z,\bar{z})
=
z^{3}\bar{z}^{3}A(z\bar{z})+(z^3+\bar z^3)B(z\bar{z}),
\end{equation}
where $A$ and $B$ are series in one variable with real coefficients:
\begin{equation*}
A(z\bar z)=\sum_{\substack{k\ge 0\\k\ne 2 \pmod 3}}a_k z^k\bar z^k\,,
\quad
B(z\bar z)=\sum_{\substack{k\ge 0\\k\ne 2 \pmod 3}}b_k z^k\bar z^k\,\,,
\end{equation*}
where $b_0=|h_{30}|$.
Moreover, the coefficients of the series $A$ and $B$ are unique.
\end{proposition}

\begin{figure}
\begin{center}
\includegraphics[width=0.6\textwidth]{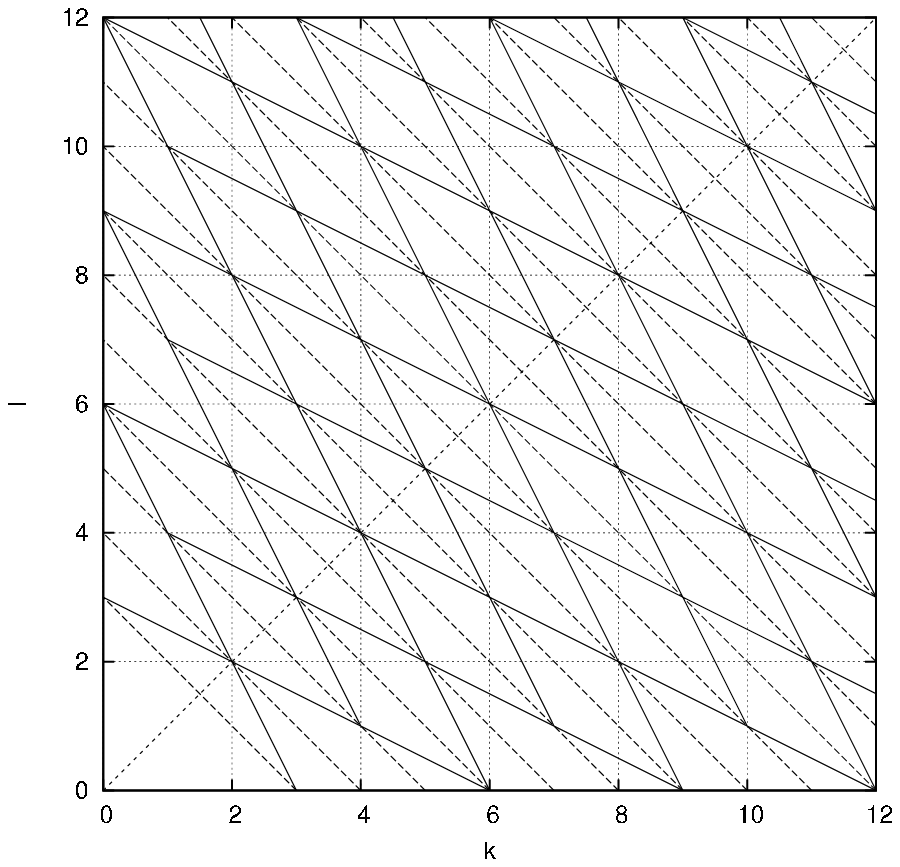}
\end{center}
\caption{Resonant terms for $n=3$ correspond to points of intersections
of solid lines. Dashed lines connects terms of equal orders.\label{Fig:n3}}
\end{figure}

\begin{proof}
Similarly to the previous section a rotation of the
coordinates makes the coefficient of the leading order real
and the Hamiltonian takes the form
\[
h(z,\bar z)=b_0(z^3+\bar z^3)+\sum_{\substack{k+l\ge4\\ k=l\pmod 3}}
h_{kl}z^k\bar z^l
\]
where $b_0=|h_{30}|$.
We group together terms of the same order and define $\mathcal H^3_m$
to be the set of real-valued homogeneous resonant polynomials of
order $m$. It is convenient to represent
the resonant terms using a diagram shown on Figure~\ref{Fig:n3}.
It can be checked by induction that
\begin{eqnarray*}
\dim\mathcal H^3_{3k}&=&k+1,\\
\dim\mathcal H^3_{3k+1}&=&k,\\
\dim\mathcal H^3_{3k+2}&=&k+1\,.
\end{eqnarray*}

The leading order of $h$ is given by
\(
h_3=b_0(z^3+\bar z^3)\,,
\)
and the homological operator has the form
\begin{equation}\label{L3}
L_3(\chi)=2i\{\chi,h_3\}=6i b_0\bar z^2\frac{\partial \chi}{\partial z}-
6i b_0 z^2\frac{\partial \chi}{\partial \bar z}\,.
\end{equation}
This formula implies that $L(\mathcal H^3_m)\subset\mathcal H^3_{m+1}$.
In order to study the normal form we need a description of the complement
to the image. Uniqueness properties are related to the properties of
the kernel. A standard result from Linear Algebra provides
the following relation:
\[
\mbox{co-dim\,Image}(L_3)=
\dim\mathcal H^3_{m+1}-\dim\mathcal H^3_m+\dim\ker L_3\,.
\]
Obviously $L_3(h_3^k)=2i\{h_3^k,h^3\}=0$ for any $k\in\mathbb N$
and therefore the kernel of $L_3$ restricted on $\mathcal H^3_{3k}$
is not trivial.
The results of the study of $L_3$ are summarised in Table~\ref{Ta:n3}.

\begin{table}
\begin{center}
\begin{tabular}{|c|c|c|c|c|}
\hline
$m$ & $\dim\mathcal H^3_m$ & $\dim\mathcal H^3_{m+1}$ & $\dim\ker L_3$ &
$\mathrm{co{-}dim\ Image}\,L_3$ \\
\hline
$3k$   & $k+1$ & $k$ &   $1$ & $0$ \\
$3k+1$ & $k  $ & $k+1$ & $0$ & $1$ \\
$3k+2$ & $k+1$ & $k+2$ & $0$ & $1$   \\
\hline
\end{tabular}
\end{center}
\caption{Properties of the homological operator for $n=3$.\label{Ta:n3}}
\end{table}

It is convenient to write a resonant monomial of order $m$ in the form
\[
Q_{mj}:=
z^{(m+3j)/2}\bar z^{(m-3j)/2}
\]
where $-\ipart{\frac m3}\le j \le \ipart{\frac m3}$ and $j=m\pmod2$.
Therefore for a fixed $m$ the index $j$ changes with step $2$.
A direct substitution to the definition of $L_3$ shows
\begin{eqnarray*}
L_3(Q_{mj})&=&
6i b_0 \frac{m+3j}{2} z^{(m+3j)/2-1}\bar z^{(m-3j)/2+2}
\\&&
-
6i b_0 \frac{m-3j}{2} z^{(m+3j)/2+2}\bar z^{(m-3j)/2-1}
\\&=&
6i b_0 \frac{m+3j}{2} Q_{m+1,j-1}
-
6i b_0 \frac{m-3j}{2} Q_{m+1,j+1}\,.
\end{eqnarray*}
The action of $L_3$ is represented in the diagram shown on Figure~\ref{Fig:dia3}.
We see that for all $m$ the matrix of $L_3$ is two diagonal. Analysing the
cases of $m=3k$, $m=3k+1$ and $m=3k+2$ separately, we see that since $b_0\ne0$
the rank of the matrix is maximal and equals to $k$, $k$ and $k+1$ respectively.
Consequently,
if $m\ne0\pmod 3$ the kernel is empty, and if $m=0\pmod 3$ the kernel is
one dimensional and generated by $h_3^{k}$.

\begin{figure}
\begin{center}
\includegraphics[width=0.7\textwidth]{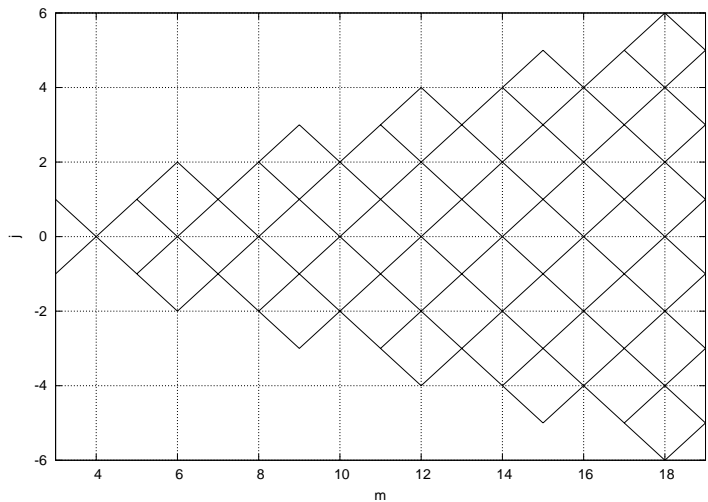}
\end{center}
\caption{Action of the operator $L_3:\mathcal H^3_m\to \mathcal H^3_{m+1}$ 
on monomials.
Each monomial is connected by a line (or two lines) to its image. \label{Fig:dia3}}
\end{figure}

We see that the image of $L_3$ completely covers $\mathcal H^3_{p}$ with $p=1\pmod 3$
and has one dimensional complement otherwise. Taking into account the structure
of the matrix of $L_3$ we see that the complements are generate either by $Q_{p,0}$
if $p$ is even, or by $Q_{p,1}+Q_{p,-1}$ if $p$ is odd.

Now we follow the same strategy we used in the previous two sections.
We construct inductively a sequence of substitutions:
\[
\tilde h=\exp(L_{\chi_{p-1}})h=h+L_3(\chi_{p-1})+O_{2p-3}
\]
and choose $\chi_{p-1}$ in such a way that $[\tilde h]_p$ is in the complement space
to $L_3(\mathcal H^3_{p-1})\subset\mathcal H^3_p$. We see that $[\tilde h]_p=0$
if $p=1\pmod 3$, and provided $p\ne1\pmod3$ we get
\begin{eqnarray*}
[\tilde h]_p&=&a_k z^{k+3}\bar z^{k+3}\qquad\qquad\mbox{for $p=2k+6$}\,,
\\{}
[\tilde h]_p&=&b_k z^{k}\bar z^k(z^{3}+\bar z^3)
\qquad\mbox{for $p=2k+3$,}
\end{eqnarray*}
for some $a_k,b_k\in\mathbb R$.
We have chosen $k$ in such a way that $k=0$ corresponds to the lowest
non-zero order. Indeed, the lowest odd order in $\tilde h$
is obviously 3, and the lowest
even order is $6$ and not $4$ because $4=1\pmod3$.

Repeating the argument inductively, we show that $h$ can be transformed to
the form~(\ref{Eq:complexnf3}). Uniqueness follows from the fact that
the kernel of $L_3$ is generated by powers of $h_3$ only
and we omit it since it repeats literally arguments
from the proofs of Propositions~\ref{Pro:n5} and~\ref{Pro:n4}.
\end{proof}

\section{Quasi-resonant normal forms\label{Se:quasires}}

In this section we prove Theorem \ref{Thm:qr}.
Instead of the individual map $F_0$ we consider an analytic family of
area-preserving maps $F_\varepsilon$, which coincides with $F_0$ at $\varepsilon=0$.
Without loosing in generality we assume that for all $\varepsilon$
the fixed point is at the origin. Then in the complex variables $(z,\bar z)$
the first component of $F_\varepsilon$ can be written in the form of a Taylor series:
\[
f_\varepsilon (z, \bar z)= \mu z+
\sum_{\substack{k+l+j\ge2\\ k,l,j\ge0}}
c_{klj}z^k\bar z^l \varepsilon^j \,,
\]
where $\mu$ is the multiplier of $F_0$.
We see that the series involves three variables
$(z,\bar z, \varepsilon )$ instead of two variables $(z,\bar z)$.
Moreover, the classical normal form theory implies that there is a formal
change of variables which eliminates all non-resonant terms from this sum.
We will assume that this change of variables has been done.

As in the Theorem \ref{Theorem_Interpol} it can be shown that there exists
a unique real-valued formal Hamiltonian
\begin{equation} \label{heps}
h(z,\bar z; \varepsilon )=\sum_{\substack{k+l+j\ge3\\ k,l,j\ge0}}
h_{klj}z^k\bar z^l \varepsilon^j \,
\end{equation}
such that
\begin{equation*}
f_\varepsilon=\mu\exp(L _{h_\varepsilon})z.
\end{equation*}
We note that the sum in (\ref{heps}) contains only resonant terms $k=l\pmod n$
(assuming $\mu^n=1$).
The proof of this statement is similar to the proof provided
in Section \ref{Se:fi} for the case of an individual map.
Of course, one should consider homogeneous polynomials in three variables.

\begin{proposition}\label{Pro:quasi}
If $h_{n00}\ne0$, there is a formal canonical change
of variables which transforms $h(z,\bar z; \varepsilon )$
into
\begin{equation}
\widetilde{h}(z,\bar z; \varepsilon )
=
z\bar{z}A(z\bar{z}; \varepsilon)+(z^n+\bar z^n)B(z\bar{z}; \varepsilon),
\end{equation}
where $A$ and $B$ are series in two variables with real coefficients:
\begin{itemize}
\item
if $n\ge4$ and $h_{220}h_{n00}\ne0$
\begin{equation}
A(z\bar{z},\varepsilon)=\sum_{km\ge0}a_{km} z^k\bar z^k\,\varepsilon^m ,\quad
B(z\bar{z},\varepsilon)=\sum_{km\ge0}b_{km}z^k\bar z^k\varepsilon ^m\,, 
\quad a_{00}=0,
\end{equation}
\item
if $n=3$ and $h_{300}\ne0$
\begin{eqnarray}
A(I,\varepsilon)&=&
\sum_{\substack{km\ge0\\ k\ne 1\pmod 3}}a_{km}z^k\bar z^k\varepsilon^m\,,\qquad
a_{00}=a_{10}=0\,,\\
B(I,\varepsilon)&=&
\sum_{\substack{km\ge0\\ k\ne 2\pmod 3}}b_{km}z^k\bar z^k\varepsilon^m\,,
\end{eqnarray}
\end{itemize}
where $b_{00}=|h_{n00}|$.
Moreover the coefficients of the series $A$ and $B$ are unique.
\end{proposition}

\begin{proof}%[Proof of Theorem~\ref{Thm:qr}]
The scheme of the proof is similar to the previous sections.
The unique normal form is constructed by induction: we use
a sequence of canonical transformations to eliminate as many
resonant terms as possible.
As in the previous sections we first use the rotation
$z\mapsto e^{-i\arg(h_{n00})/n }\,z$ which transforms
\[
h_{n00}z^n + h_{0n0} \bar z^{n} \mapsto b_{0}(z^n+\bar z^{n}),
%\label{n=4_LeadingOrder}
\]
where $b_{00}=|h_{n00}|$.

First we consider the case of $n\ge 4$. The terms in (\ref{heps}) can be grouped
in the following way:
\[
%\chi_p=\sum_{m=0}^{p-2}\varepsilon^m \chi_{p-m,m}
h(z,\bar z; \varepsilon )=
\sum_{s\ge2}\sum_{j=0}^{s-1}\varepsilon^j h_{s-j,j}
\]
where $h_{s-j,j}\in\mathcal H^n_{s-j}$. We note that
after the rotation the terms with $s=2$ already have the
desired form:
\[
h_{2,0}= b_{00}(z^n+\bar z^{n})\qquad\mbox{and}\qquad h_{1,1}=h_{111}z\bar z \, .
\]
and we simply let $a_{11}=h_{111}$.
We will simplify the terms of the formal series in the following order:
for each fixed $s$ starting with $s=3$, we will run $j$ from $0$ to $s-1$.
Following this order, we perform a sequence of canonical
coordinate changes generated by an auxiliary Hamiltonian
$\chi_{p-k,k} \in\mathcal H^n_{p-k}$:
\[
(z,\bar z) \mapsto \Phi ^{\varepsilon^k}_{\chi_{p-k,k}}(z,\bar z)
\,.
\]
This is a canonical change of variables and the Hamiltonian $h$ is transformed into
\[
\tilde h=\exp\left(\varepsilon^k \chi_{p-k,k}\right) h\,.
\]
The function $\chi_{p-k,k}$ will be chosen to normalize the
term $\varepsilon^k h_{p-k+1,k}$. Writing Lie series for the
transformed Hamiltonian we get
\[
\tilde h=h+\varepsilon^kL_{\chi_{p-k,k}}h+
\sum_{l\ge2}\varepsilon^{kl}\frac1{l!}L^l_{\chi_{p-k,k}}h\,.
\]
It is easy to check that
\[
\tilde h_{s-j,j}=h_{s-j,j}
\]
for $s \le p$ $(0 \le j \le s-1)$
and for $s=p+1, j<k$.
For $s=p+1, j=k$ we get
\begin{equation}\label{he}
\tilde h_{p-k+1,k}=h_{p-k+1,k}+{\left[ L_{\chi_{p-k,k}}h_{2,0} \right]}_{p-k+1}.
 %=h_{p-k+1,k}+L_2\chi_{p-k,k} \, ,
\end{equation}
Since $h_{2,0}=h_{2,2,0} z^2 \bar z^2 + b_{00}(z^n+\bar z^{n})$ agrees with
$h_2$ in (\ref{Eq:h2}) for $n>4$ or (\ref{n=4_LeadingOrder}) for $n=4$
we can rewrite the formula using the homological operator $L_{p-k}$ ($L_{p-k}=L_2$
 for $n=4$):
\[
\tilde h_{p-k+1,k}=h_{p-k+1,k}+L_{p-k}\chi_{p-k,k} \, .
\]
The explicit description
for the complements of the homological operators was provided
in Sections \ref{Se:n5}  and \ref{Se:n4} for $n\ge5$ and $n=4$ respectively.
So there exists $\chi_{p-k,k} \in\mathcal H^n_{p-k}$ such that $\tilde h_{p-k+1,k}$
takes the form
(\ref{Eq:podd}) if $p-k$ is odd and (\ref{Eq:peven}) if $p-k$ is even.

In the case of the third order resonance the Hamiltonian (\ref{heps})
can be written as
\[
h(z,\bar z; \varepsilon )=
\sum_{s\ge3}\sum_{j=0}^{s-1}\varepsilon^j h_{s-j,j} \, ,
\]
where $h_{s-j,j}\in\mathcal H^3_{s-j}$. Then we continue in the same way
as in the case of $n\ge 4$ but Equation (\ref{he}) is replaced with
\[
\tilde h_{p-k+1,k}=h_{p-k+1,k}+ L_{\chi_{p-k,k}}h_{3,0} .
\]
Since $h_{3,0}=h_3$ then $L_{\chi_{p-k,k}}h_{3,0} = L_3(\chi_{p-k,k})$,
where the homological operator $L_3$ is defined by (\ref{L3}).
Using the results of Section \ref{Se:n3} we get
\[
\tilde h_{p-k,k}=0 \quad {\rm for \ }  p-k=1\pmod3
\]
and if $p-k \ne 1\pmod3$
\begin{eqnarray*}
[\tilde h]_{p-k,k}&=&a_{l,k} z^{l+1}\bar z^{l+1}\qquad\qquad\mbox{for $p-k=2l+6$; $a_{0,0}=a_{1,0}=0$}\,,
\\{}
[\tilde h]_{p-k,k}&=&b_{l,k} z^{l}\bar z^l(z^{3}+\bar z^3)
\qquad\mbox{for $p-k=2l+3$.}
\end{eqnarray*}
%After the reversion to the coordinates $(I, \varphi )$ the Hamiltonian takes the form (\ref{Heps}).

The proof of the uniqueness is very similar to the previous sections. Suppose
$h$ can be transformed into two different simplified normal forms
$\tilde h$ and $\tilde h'$. The formal Hamiltonian $\chi$ in (\ref{Eq:subst})
now takes the form:
\[ \chi = \sum_{s\ge p}\sum_{j=0}^{s-1}\varepsilon^j \chi_{s-j,j} \]
and the lowest order is given by the first non-zero term $\varepsilon^k \chi_{p-k,k}$.
Then  \[\tilde h_{s-j,j}=\tilde h'_{s-j,j} \ {\rm for} \  s \le p \ (0 \le j \le s-1) \
{\rm and} \ {\rm for} \  s=p+1, j<k\] and
%coincide up to
\[
\tilde h'_{p-k+1,k}=\tilde h_{p-k+1,k}+L_{p-k}\chi_{p-k,k} \, ,
\]
where $L_{p-k}$ is the homological operator
($L_{p-k}=L_2$ for $n=4$ and $L_{p-k}=L_3$ for $n=3$).
As it was shown in Sections \ref{Se:n5} and \ref{Se:n4}
the knowledge of the kernel of the  homological operator
allows us to show the existence of $\tilde \chi$ which conjugates $\tilde h$
and  $\tilde h'$ but starts at least from the next order in comparison with $\chi$.
Consequently $\tilde h'_{p-k+1,k}=\tilde h_{p-k+1,k}$ and we conclude
$\tilde h=\tilde h'$ by induction.

\end{proof}

\end{document}